\documentclass[12pt]{article}
\usepackage[english]{babel}
\usepackage{graphicx}
\usepackage{amsmath}
\usepackage{amsfonts}
\usepackage{amssymb}
\usepackage{amsthm}
\usepackage{hyperref}
\usepackage[T1]{fontenc}
\usepackage{color}
\usepackage{mathtools}
\usepackage{comment}
\usepackage{enumerate}

\input xy
\xyoption{all}

\newtheorem{thm}{Theorem}[section]
\newtheorem{cor}[thm]{Corollary}
\newtheorem{lem}[thm]{Lemma}
\newtheorem{pro}[thm]{Proposition}

\theoremstyle{definition}
\newtheorem{defi}[thm]{Definition}

\newtheoremstyle{remarque}{}{}{}{}{\it}{.}{\newline}{}
\theoremstyle{remarque}
\newtheorem*{rem}{Remark}
\newtheorem*{rems}{Remarks}

\newcommand{\asd}[5]{%
\setbox1=\hbox{\ensuremath{^{#1}}}%
\setbox2=\hbox{\ensuremath{_{#2}}}%
\setbox5=\hbox{\ensuremath{#5}}%
\hspace{\ifnum\wd1>\wd2\wd1\else\wd2\fi}%
\ensuremath{\copy5^{\hspace{-\wd1}\hspace{-\wd5}#1\hspace{\wd5}#3}%
_{\hspace{-\wd2}\hspace{-\wd5}#2\hspace{\wd5}#4}%
}}

\usepackage[OT2,T1]{fontenc}
\DeclareSymbolFont{cyrletters}{OT2}{wncyr}{m}{n}
\DeclareMathSymbol{\Sha}{\mathalpha}{cyrletters}{"58}
\DeclareMathSymbol{\Brusse}{\mathalpha}{cyrletters}{"42}

\newcommand{\id}{\mathrm{id}}
\renewcommand{\ker}{\mathrm{Ker }}
\newcommand{\im}{\mathrm{Im}}

\newcommand{\ext}{\mathrm{Ext}}

\newcommand{\res}{\mathrm{Res}}
\newcommand{\aut}{\mathrm{Aut}}
\newcommand{\out}{\mathrm{Out}}
\newcommand{\saut}{\mathrm{SAut}}
\newcommand{\sout}{\mathrm{SOut}}
\renewcommand{\int}{\mathrm{int}}
\newcommand{\Int}{\mathrm{Int}}

\newcommand{\gm}{\mathbb{G}_{\mathrm{m}}}

\newcommand{\gal}{\mathrm{Gal}}

\newcommand{\red}{\mathrm{red}}
\newcommand{\spec}{\mathrm{Spec}\,}

\newcommand{\bb}[1]{\mathbb{#1}}

\newcommand{\f}{\mathfrak{f}}
\newcommand{\g}{\mathfrak{g}}
\newcommand{\z}{\mathfrak{z}}

\usepackage{marvosym}

\title{Extensions of algebraic groups with finite quotient and nonabelian 2-cohomology}
\author{Giancarlo Lucchini Arteche\\[5mm]
{\it\small Centre de Math\'ematiques Laurent Schwartz, UMR CNRS 7640}\\
{\it\small \'Ecole polytechnique, 91128 Palaiseau, France}\\
{\small giancarlo.lucchini-arteche@polytechnique.edu }
}

\date{}

\numberwithin{equation}{subsection}

\begin{document}

\maketitle

\begin{abstract}
For a finite smooth algebraic group $F$ over a field $k$ and a smooth algebraic group $\bar G$ over the separable closure of $k$, we define the notion of $F$-kernel in $\bar G$ and we associate to it a set of nonabelian 2-cohomology. We use this to study extensions of $F$ by an arbitrary smooth $k$-group $G$. We show in particular that any such extension comes from an extension of finite $k$-groups when $k$ is perfect and we give explicit bounds on the order of these finite groups when $G$ is linear. We prove moreover some finiteness results on these sets.\\

{\bf Keywords :} Algebraic groups, Non abelian cohomology.

{\bf MSC classes (2010):} 20J06, 18G50, 14L99.
\end{abstract}

\section{Introduction}
It is a well known fact in abstract group theory that group extensions
\[1\to G\to E\to F\to 1,\]
where $G$ is an abelian group and $F$ is an arbitrary group, are classified by what are called \emph{factor systems}. In the modern language of cohomology theory, these factor systems turn out to be 2-cocycles for the group action of $F$ on $G$ given by conjugation in $E$. Extensions like the one above are thus classified by the group cohomology set $H^2(F,G)$ (cf.~for example \cite[IV.4]{MacLane}). In a nonabelian context (that is, when $G$ is also arbitrary), one loses the action of $F$ over $G$ given by conjugation in $E$ but can still define the notion of \emph{$F$-kernel} (cf.~\cite[IV.8]{MacLane}) which allows a description of the set of extensions.\\

Such a description can be generalized to the case of topological groups if one is careful enough to impose the good conditions on the 2-cocycles one works with, a natural condition being for instance continuity of all considered maps. This was already done for example by Springer in \cite{SpringerH2}, where he considers the action of the absolute Galois group of a perfect field $k$ on the $k_s$-points of an algebraic group over $k$, where $k_s$ denotes a separable closure of $k$, obtaining thus the notion of \emph{$k$-kernel}. These ideas were later resumed (and treated with much more precision and care) by Borovoi in the case of linear groups over a field of characteristic 0 (cf.~\cite{Borovoi93}) and by Flicker, Scheiderer and Sujatha in the more general case of smooth $k$-groups for arbitrary $k$ (cf. \cite{FSS}). With such tools one can prove interesting results, as for example Springer's result stating the existence, for an arbitrary homogeneous space under an algebraic group over a perfect field of cohomological dimension $\leq 1$, of a principal homogenenous space surjecting onto it; or Borovoi's abelianization of the nonabelian Galois cohomology and its known consequences in the arithmetic study of homogeneous spaces.\\

However, if one wishes to extend such theories to algebraic groups \emph{acting on other algebraic groups} (having thus in mind studying their extensions), serious problems appear. For starters, an algebraic group, when considered as a group scheme over the base field $k$, brings with it not a single group, but infinitely many (one for each scheme over $k$), making it difficult to define the notion of kernel (let alone that of 2-cocycle) in this context in a ``naive'' way (see however \cite{DemarcheHochschild}). These difficulties can be overcome in the case where the group that acts is a \emph{finite} smooth $k$-group. This is done by a slight generalization of the nonabelian Galois cohomology cited above.

The interest for doing so is not only based on the simple question of understanding algebraic group extensions, but also on the fact that one can use these tools in order to exhibit finite subgroups of arbitrary algebraic groups which are defined over the base field and intersect every connected component of the group. This has already been used by the author in the study of Brauer groups of homogeneous spaces with non connected stabilizer (cf.~\cite{GLABrnral2}), but also for example by Gille and Reichstein (cf.~\cite{GilleReichstein}) and by L\"otscher, MacDonald, Meyer and Reichstein (cf.~\cite{Reichstein et cie}) in the study of essential dimension for linear algebraic groups. In the latter, an important issue is to control the prime numbers dividing the order of the finite group obtained. We thus try to control this order in our main result in this direction (Theorem \ref{theoreme reduction ext}).

The existence of such subgroups had already been stated by Borel and Serre for a perfect field $k$, although they only gave the proof for linear $G$ and $k=k_s$ of characteristic zero (cf.~\cite[Lem.~5.11 and footnote on p.~152]{BorelSerre}). This result was extended shortly after by Platonov to the case of a perfect field, but still for linear groups (cf.~\cite[Lem.~4.14]{Platonov}). Finally, the same assertion has also been recently proved by Brion in an even more general setting (cf.~\cite[Thm.~1]{Brion}), although with no explicit bounds on the order of the finite groups thus obtained.\\

The structure of this article is then as follows:

In section \ref{section}, we recall the basic theory of kernels and nonabelian 2-cohomology (\S\ref{section k-liens}) as well as that of actions by group automorphisms (\S\ref{section k-actions}). We define then the notion of an $F$-kernel in $\bar G$ for a finite smooth $k$-group $F$ and a smooth $k_s$-group $\bar G$ (\S\ref{section F-liens}). This allows us to define the notion of an \emph{outer action} of $F$ on a smooth $k$-group $G$ as a particular type of $F$-kernel that appears naturally when studying extensions of $F$ by $G$. Later, we associate to an $F$-kernel a nonabelian 2-cohomology set (\S\ref{section H2 et ext}) and, in the case of an outer action, we compare this cohomology set with the set of extensions of $F$ by $G$ (\S\S\ref{section twists}--\ref{section gps commutatifs}). Finally we describe the extensions by $G$ via its smooth center (\S\ref{section H2L et H2Z pour F-liens}). This last subsection is not used in the proof of our main theorem and hence it can be skipped in the first lecture.

In section \ref{section reduction}, we prove our main result on the ``reduction'' of extensions (Theorem \ref{theoreme reduction ext}). More precisely, given any extension of $k$-groups
\[1\to G \to H\to F\to 1,\]
there exists a finite smooth $k$-subgroup $S$ of $G$ and an extension of $F$ by $S$ pulling back the original extension. Moreover, under further hypotheses, the result gives an explicit bound on the order of $S$.

Finally, in section \ref{section finitude}, we study the finiteness of the set of extensions of $F$ by $G$. It turns out that this set is actually always finite when $k$ is a finite field and, when $G$ is linear, this is quite often the case also for other fields (Theorem \ref{theoreme finitude}). Here, the description given in \S\ref{section H2L et H2Z pour F-liens} is important for the proof of these results.

\paragraph*{Acknowledgements}
The author would like to warmly thank Michel Brion for the intense email conversations and for his constant interest and support on this work, as well as Zinovy Reichstein, Jean-Pierre Serre and Roland L\"otscher for interesting discussions and for pointing out references to previous work on this subject.

This work was partially supported by the Fondation Math\'ematique Jacques Hadamard through the grant N\textsuperscript{o} ANR-10-CAMP-0151-02 in the ``Programme des Investissements d'Avenir''.

\section{Extensions of algebraic groups and nonabelian 2-cohomology}\label{section}
Let $k$ be a field, $k_s$ be a separable closure and $\Gamma=\gal(k_s/k)$ be the absolute Galois group. We wish to study extensions of a finite smooth algebraic $k$-group $F$ by another smooth algebraic $k$-group $G$. In this section, we define a set of nonabelian 2-cohomology that ``almost'' classifies extensions of $F$ by $G$ in this setting. Here, ``almost'' means that there may be classes of extensions that fall into the same element of this set. This shortcoming can be overcome (see \S\ref{section comparaison} below): we will see in fact that two classes of extensions will fall into the same element if and only if one of them is a very particular twist of the other one.\\

For \emph{any} $k$-group $F$, it is clear that we get an action of $\Gamma$ on $F(k_s)$. Throughout the text, we will use the notation $F_\Gamma$ for the semi-direct product $F(k_s)\rtimes\Gamma$ and we will denote by $\gamma_F$ the natural splitting $\Gamma\to F_\Gamma$. An arbitrary element of $F_\Gamma$ is written $(f,\sigma):=f\cdot\gamma_F(\sigma)$, with $f\in F(k_s)$ and $\sigma\in\Gamma$.

\subsection{$k$-kernels and $\Gamma$-kernels}\label{section k-liens}
We fix here some notations we will be using in the definition of $F$-kernels below. For details on all the notions mentioned in this section, we send the reader to \cite{SpringerH2}, \cite{Borovoi93} and \cite{FSS}. Some familiarity with these references is expected from the reader in order to follow the constructions in the subsequent sections. Mostly with the group cohomology presented in \cite[1.12--1.18]{SpringerH2} and the Galois cohomology presented in \cite[\S1]{FSS}\\

Let $\bar G$ be a smooth algebraic group over $k_s$. We denote by $\aut(\bar G)$ its group of automorphisms of $k_s$-group schemes. Consider now the group of \emph{$k$-semialgebraic automorphisms} (cf.~\cite[1.2]{FSS}), denoted $\saut(\bar G/k)$, or simply $\saut(\bar G)$ if $k$ is implicit. This group fits into an exact sequence
\begin{equation}\label{equation suite exacte avec SAut G}
1\to\aut(\bar G)\to\saut(\bar G)\rightarrow\Gamma,
\end{equation}
where $\Gamma$ denotes the absolute Galois group $\gal(k_s/k)$. We denote by $\Int(\bar G)$ the subgroup of $\aut(\bar G)$ given by inner automorphisms. This is a normal subgroup of both $\aut(\bar G)$ and $\saut(\bar G)$. Define then
\begin{align*}
\out(\bar G)&:=\aut(\bar G)/\Int(\bar G),\\
\sout(\bar G)=\sout(\bar G/k)&:=\saut(\bar G/k)/\Int(\bar G).
\end{align*}
Sequence (\ref{equation suite exacte avec SAut G}) gives then the following exact sequence
\begin{equation}\label{equation suite exacte avec SOut G}
1\to\out(\bar G)\to\sout(\bar G)\xrightarrow{q}\Gamma.
\end{equation}

Recall the definition of a $k$-kernel (cf.~\cite[1.11]{FSS}):

\begin{defi}\label{definition k-lien}
A \emph{$k$-kernel} in $\bar G$ is a group homomorphism $\kappa:\Gamma\to\sout(\bar G)$ such that
\begin{enumerate}
\item[$(i)$] $\kappa$ splits sequence (\ref{equation suite exacte avec SOut G}), i.e. $q\circ\kappa$ is the identity on $\Gamma$, and
\item[$(ii)$] there exists a section $\f:\Gamma\to\saut(\bar G)$ of \eqref{equation suite exacte avec SAut G} lifting $\kappa$ that is continuous in the sense of \cite[1.10]{FSS}.
\end{enumerate}
A pair $(\bar G,\kappa)$ as above is simply called a $k$-kernel. It will be said to be \emph{trivial} if there exists such an $\f$ which is moreover a splitting of \eqref{equation suite exacte avec SAut G}.
\end{defi}

For any smooth algebraic $k$-group $G$ there is a trivial $k$-kernel in $G_{k_s}$. Indeed, \cite[1.4]{FSS} tells us that such a group comes with a natural splitting $\f_G$ of \eqref{equation suite exacte avec SAut G} for $\bar G=G_{k_s}$. This induces a trivial kernel that we will denote by $\kappa_G$. Conversely (cf.~\cite[1.15]{FSS}), every splitting of \eqref{equation suite exacte avec SAut G} induces a $k$-form of $\bar G$. Note that a trivial kernel may have more than one associated $k$-form.\\

Back to a general smooth $k_s$-group $\bar G$, there is a group homomorphism (cf.~\cite[1.3]{FSS})
\begin{equation}\label{equation morphisme r}
r:\saut(\bar G)\to\aut(\bar G(k_s)),
\end{equation}
which may not be injective in general, as it is the case for finite algebraic groups. However, since $\ker(r)$ and $\Int(\bar G)$ are both normal in $\saut(\bar G)$ and have trivial intersection, one can easily see that the subgroup generated by them is isomorphic to their direct product. In particular, we see that $r$ factors through the quotient by $\Int(\bar G)$, giving a homomorphism $\bar r:\sout(\bar G)\to\out(\bar G(k_s))$. This tells us that a $k$-kernel in $\bar G$ is in particular (via $\bar r\circ\kappa$) a $\Gamma$-kernel in $\bar G(k_s)$ in the sense of \cite[1.12]{SpringerH2}. From now on, we will always omit $\bar r$ and simply denote by $\kappa$ the corresponding $\Gamma$-kernel.

For a given $k$-kernel $(\bar G,\kappa)$, one may consider the set $Z^2(k,\bar G,\kappa)$ of $2$-cocycles and the cohomology set $H^2(k,\bar G,\kappa)$ is then defined as the set of equivalence classes of $2$-cocycles, cf.~\cite[1.17]{FSS}. One can also consider the corresponding sets $Z^2(\Gamma,\bar G(k_s),\kappa)$ and $H^2(\Gamma,\bar G(k_s),\kappa)$ in the context of group cohomology, as defined in \cite[1.14--1.15]{SpringerH2}. The latter are defined as follows: if $G(k_s)$ is considered as a discrete group and $\aut(\bar G(k_s))$ is given the weak topology with respect to the evaluation maps, then an element of $Z^2(\Gamma,\bar G(k_s),\kappa)$ is a pair $(\f,\g)$ of continuous maps
\[\f:\Gamma\to\aut(\bar G(k_s))\quad\text{and}\quad\g:\Gamma\times\Gamma\to \bar G(k_s),\]
satisfying the following relations, for $\sigma,\tau,\upsilon\in\Gamma$,\footnote{These are the relations given in \cite[1.5]{Borovoi93}, which differ by a sign from the ones in \cite{SpringerH2} and \cite{FSS}.}
\begin{gather}
\f_\sigma\!\!\!\mod\Int(\bar G)=\kappa_\sigma,\label{equation proprietes 2 cocycles cong avec kappa}\\
\f_{\sigma\tau}=\int(\g_{\sigma,\tau})\circ \f_\sigma\circ \f_\tau,\label{equation proprietes 2 cocycles f morphisme a intg pres}\\
\g_{\sigma,\tau\upsilon}\f_\sigma(\g_{\tau,\upsilon})=\g_{\sigma\tau,\upsilon}\g_{\sigma,\tau}.\label{equation proprietes 2 cocycles eq 2 cocycles}
\end{gather}
A cocycle $(\f,\g)$ is equivalent to $(\f',\g')$ if there exists a continuous map $c:\Gamma\to G(k_s)$ such that, for $\sigma,\tau\in\Gamma$,
\begin{align}
\f'_\sigma &=(c\cdot \f)_\sigma :=\int(c_\sigma)\circ \f_\sigma,\label{equation equivalence 2-cocycles f}\\
\g'_{\sigma,\tau} &=(c\cdot g)_{\sigma,\tau} :=c_{\sigma\tau}\g_{\sigma,\tau}\f_\sigma(c_\tau)^{-1}c_{\sigma}^{-1}.\label{equation equivalence 2-cocycles g}
\end{align}

A priori we should follow \cite{FSS} and not \cite{SpringerH2} when working with Galois cohomology. However, from what we said above about the arrow \eqref{equation morphisme r} we also deduce that, given elements in $\aut(\bar G(k_s))$ and $\sout(\bar G)$ with the same image in $\out(\bar G(k_s))$, there is a \emph{unique} element in $\saut(\bar G)$ lifting both of them. Thus, by equation \eqref{equation proprietes 2 cocycles cong avec kappa}, one sees that for a given cocycle $(\f,\g)$ as above we may assume that $\f$ takes values in the group $\saut(\bar G)$ in a canonical way. Using \cite[1.13]{FSS}, one deduces a bijection between 2-cocycles for a given $k$-kernel in the sense of \cite[1.17]{FSS} and 2-cocycles for the corresponding $\Gamma$-kernel in the sense of \cite[1.14--1.15]{SpringerH2}. We will hence identify these two sets hereafter, as well as the correponding $H^2$ sets, and work in the group cohomology setting.

Recall finally that the subset $N^2(k,\bar G,\kappa)$ of neutral elements of $H^2(k,\bar G,\kappa)$ is defined as those classes that are represented by a cocycle $(\f,\g)$ such that $\g=1$.

\subsection{Actions by group automorphisms}\label{section k-actions}
Recall that a $k$-action of a $k$-group $G$ on a $k$-scheme $V$ is a $k$-morphism $a:G\times_k V\to V$ such that the following diagrams commute
\begin{equation}\label{diagramme k action}
\xymatrix{
G\times_k G\times_k V \ar[rr]^>>>>>>>>>{m_G\times\id_V} \ar[d]_{\id_G\times a} && G\times_k V \ar[d]^a \\
G\times_k V \ar[rr]^a && V,
}
\quad
\xymatrix{
\spec k\times_k V \cong V \ar[rr]^<<<<<<<<<{\id_V} \ar[d]_{e\times \id_V} && V \ar@{=}[d] \\
G\times_k V \ar[rr]^{a} && V,
}
\end{equation}
where $e$ is the neutral element in $G$ and $m_G:G\times_k G\to G$ denotes the multiplication morphism. We will
denote by $a_2$ the natural $k$-action $G\times_k V\times_k V\to V\times_k V$ of $G$ on $V\times_k V$ obtained by
acting via $a$ on each component of $V\times_k V$.

\begin{defi}\label{definition action alg}
Let $F,G$ be arbitrary algebraic $k$-groups. We will say that $F$ \emph{acts on $G$ by group automorphisms} if there is a $k$-action $a:F\times_k G\to G$ such that the following diagram commutes
\begin{equation}\label{diagramme action de groupes}
\xymatrix{
F\times_k G\times_k G \ar[rr]^>>>>>>>>>{\id_F\times m_G} \ar[d]_{a_2} && F\times_k G \ar[d]^a \\
G\times_k G \ar[rr]^{m_G} && G.
}
\end{equation}
\end{defi}

\begin{pro}\label{proposition k actions}
Let $F,G$ be \emph{smooth} algebraic $k$-groups. Denote by $\f_G$ the splitting of sequence \eqref{equation suite exacte avec SAut G} naturally associated to the trivial $k$-kernel $(G_{k_s},\kappa_G)$ (see section \ref{section k-liens}). Assume that $F$ is finite. Then the following are equivalent:
\begin{enumerate}
\item $F$ acts on $G$ by group automorphisms.
\item there is a morphism of $\Gamma$-groups $F(k_s)\to\aut(G_{k_s})$, where $\Gamma$ acts on $\aut(G_{k_s})$ by conjugation via $\f_G$,
\item there is a group homomorphism $F_\Gamma\to\saut(G_{k_s}/k)$ sending $F(k_s)$ into $\aut(G_{k_s})$ and $\gamma_F(\Gamma)$ identically to $\f_G(\Gamma)$ (see notations at the beginning of section \ref{section}).
\end{enumerate}
\end{pro}

\begin{proof}
$1\Rightarrow 2$: Let $K/k$ be a separable extension and let $f:\spec K\to F$ be a $K$-point of $F$. Consider the pullback of $a$ by this point. This is a $K$-morphism $a_f:G_K\to G_K$ which is seen to be an automorphism of $G_K$ thanks to diagram \eqref{diagramme action de groupes}. Moreover, diagram \eqref{diagramme k action} tells us that $f\mapsto a_f$ defines a group homomorphism $F(K)\to \aut_K(G_K)$. Functoriality of $a$ with respect to $K$ tells us that these homomorphisms are compatible with the action of $\Gamma$ for finite separable extensions, hence a $\Gamma$-group morphism $F(k_s)\to\aut(G_{k_s})$.\\
$2\Leftrightarrow 3$: This follows basically from the definition of a semi-direct product.\\
$2\Rightarrow 1$: We may assume that $F$ is a constant group. The existence of the $k$-action in the general case follows by Galois descent from the $\Gamma$-equivariance of the morphism (cf.~\cite[V.20]{SerreGpsAlg}). Under this assumption, the scheme $F\times_k G$ is a finite set of copies $G_f$ of $G$ indexed by the $k$-points $f$ of $F$. Hence, in order to define the $k$-morphism $a$ it will suffice to define a $k$-morphism for each copy. We send then each $G_f=G$ to $G$ via the image of $f$ in $\aut(G_{k_s})$ (which is actually a $k$-morphism by $\Gamma$-equivariance). One can easily check that this morphism satisfies diagrams \eqref{diagramme k action} and \eqref{diagramme action de groupes}.\\
\end{proof}

\subsection{$F$-kernels and outer actions}\label{section F-liens}
Let us now try to define an \emph{outer} action of $F$ on $G$, that is, an action up to inner automorphisms, by analogy with the group cohomology and the Galois cohomology settings.

Proposition \ref{proposition k actions} suggests the following definition.

\begin{defi}
Let $F$ be a finite smooth algebraic $k$-group and $\bar G$ a smooth algebraic $k_s$-group. An \emph{$(F,k)$-kernel} (or $F$-kernel, for short) in $\bar G$ is a group homomorphism $\kappa:F_\Gamma\to\sout(\bar G)$ sending $F(k_s)$ to $\out(\bar G)$ and such that the restriction of $\kappa$ to $\gamma_F(\Gamma)$ is a $k$-kernel.

For $G$ a smooth algebraic $k$-group, we define an \emph{outer action} of $F$ on $G$ to be an $F$-kernel in $G_{k_s}$ such that its restriction to $\gamma_F(\Gamma)$ is the trivial kernel $\kappa_G$ (cf.~section \ref{section k-liens}).
\end{defi}

By the same argument given in section \ref{section k-liens}, it is clear from the definition that an $F$-kernel $\kappa$ in $\bar G$ corresponds to an $F_\Gamma$-kernel in $\bar G(k_s)$ in the sense of \cite[1.12]{SpringerH2}, abusively still denoted by $\kappa$.\\

Consider now an extension of smooth algebraic $k$-groups
\[1\to G\to H\to F\to 1,\]
with finite $F$. We claim that such an extension defines an outer action of $F$ on $G$. Since $G$ is a $k$-group, we get the natural trivial $k$-kernel $\kappa_G$ in $G_{k_s}$ (cf.~section \ref{section k-liens}) given by a splitting $\f_G$ of sequence \eqref{equation suite exacte avec SAut G}. Now, since $G$ is smooth, we have $H^1_{\text{fppf}}(k_s,G)=H^1_{\text{\'et}}(k_s,G)=0$ (cf.~\cite[III.4]{Milne}) and hence $H(k_s)$ surjects onto $F(k_s)$ by the classic exact sequence in fppf cohomology. Choose then arbitrary preimages $\hat f\in H(k_s)$ for $f\in F(k_s)$ and set, for $(f,\sigma)\in F_\Gamma=F(k_s)\rtimes\Gamma$,
\[\f_{(f,\sigma)}:=\int(\hat f)\circ(\f_G)_\sigma\in\saut(G_{k_s}).\]
Define $\kappa:F_\Gamma\to\sout(G_{k_s})$ as the composition of the map $\f$ with the natural projection $\saut(G_{k_s})\to\sout(G_{k_s})$. One easily verifies then that $\kappa$ is a group homomorphism restricting to $\kappa_G$ on $\gamma_F(\Gamma)$ and hence an outer action. This $F$-kernel induced by $H$ is easily seen to be independent of the choice of the preimages.

\subsection{Nonabelian 2-cohomology and extensions}\label{section H2 et ext}
Recall that an $F$-kernel in $\bar G$ corresponds to an $F_\Gamma$-kernel in $\bar G(k_s)$. Then one can define cocycles and cohomology classes in the group-cohomological way, following \cite{SpringerH2}.

\begin{defi}\label{defintion F-cocycles}
For $\kappa$ an $F$-kernel in $\bar G$, we define $Z^2(F,\bar G,\kappa)$ as the set of pairs $(\f,\g)$ of continuous maps
\[\f:F_\Gamma\to\aut(\bar G(k_s))\quad\text{and}\quad\g:F_\Gamma\times F_\Gamma\to \bar G(k_s),\]
satisfying equations \eqref{equation proprietes 2 cocycles cong avec kappa} to \eqref{equation proprietes 2 cocycles eq 2 cocycles} for elements in $F_\Gamma$. Consequently, the set $H^2(F,\bar G,\kappa)$ is defined as the equivalence classes in $Z^2(F,\bar G,\kappa)$ using equations \eqref{equation equivalence 2-cocycles f} and \eqref{equation equivalence 2-cocycles g} for continuous maps $c:F_\Gamma\to G(k_s)$.
The set $N^2(F,\bar G,\kappa)$ is that of \emph{neutral} classes, i.e. those that are represented by a cocycle $(\f,\g)$ with $\g=1$.

If $\kappa$ corresponds to an outer action of $F$ on a smooth algebraic $k$-group $G$, then we will replace $\bar G$ by $G$ in the notations above. Still in this last case, we denote by $\ext(F,G,\kappa)$ the set of $k$-isomorphism classes of extensions of $F$ by $G$ inducing $\kappa$.
\end{defi}

Given this definition, we see that, for a smooth algebraic $k$-group $G$ with an outer action $\kappa$ of $F$, the set $H^2(F,G,\kappa)$ classifies group extensions of $F_\Gamma$ by $G(k_s)$ (cf.~\cite[1.14--1.15]{SpringerH2}), while $N^2(F,G,\kappa)$ classifies those that are split. We want to compare this set with the set $\ext(F,G,\kappa)$.\\

Consider then a class $\xi\in\ext(F,G,\kappa)$ and an extension
\[1\to G\to H_\xi \to F\to 1,\]
representing $\xi$. As in last section, by smoothness of $G$, we get the same exact sequence at the level of $k_s$-points
\[1\to G(k_s)\to H_\xi(k_s) \to F(k_s)\to 1,\]
where the homomorphisms are $\Gamma$-equivariant. Consider now the extension
\[1\to H_\xi(k_s)\to E_\xi\to \Gamma\to 1,\]
given by the semi-direct product $E_\xi:=H_\xi(k_s)\rtimes\Gamma$ associated to the natural action of $\Gamma$ on
$H_\xi(k_s)$. It is clear that the subgroup $G(k_s)$ of $H_\xi(k_s)$ is normal in $E_\xi$. Therefore, we get an extension
\[1\to G(k_s)\to E_\xi\to F_\Gamma\to 1,\]
which defines an $F_\Gamma$-kernel in $G(k_s)$ (cf.~\cite[1.13]{SpringerH2}). Now it is easy to verify that this $F_\Gamma$-kernel is $\kappa$. We get then a class $\varphi(\xi)\in H^2(F,G,\kappa)$ which is clearly independent of the choice of the extension $H_\xi$ representing $\xi$. We have thus defined a map
\[\varphi:\ext(F,G,\kappa)\to H^2(F,G,\kappa),\]
which clearly sends split extensions to neutral classes.

\subsection{Twisting by torsors on the smooth center of $G$}\label{section twists}
In order to understand which extensions of $F$ by $G$ map to the same class in $H^2(F,G,\kappa)$ by this comparison map $\varphi$, we describe here an action of the group $H^1(k,Z)$ on $\ext(F,G,\kappa)$, where $Z$ is the \emph{smooth} center of $G$.

\begin{defi}
Let $G$ be a smooth algebraic $k$-group. We define the \emph{smooth center} $Z$ of $G$ as the largest smooth central $k$-subgroup of $G$.
\end{defi}

\begin{rem}
When $k$ is perfect, if we denote by $Z'$ the schematic center of $G$, then the smooth center $Z$ corresponds to the reduced subscheme $Z'_\red$. Otherwise it corresponds to the subscheme of $Z$ given by \cite[Lem.~C.4.1]{PseudoRedGps}.
\end{rem}

Consider again an arbitrary element $\xi\in\ext(F,G,\kappa)$ and its associated extension
\[1\to G\to H_\xi \to F\to 1.\]
The smooth center $Z$ acts on $H_\xi$ naturally by conjugation and thus, if we consider a
class $\alpha\in H^1(k,Z)$ and a 1-cocycle $\z$ representing the image of this class in $H^1(k,H_\xi)$, we can consider the
twisted group $\asd{}{\z}{}{\xi}{H}$ (see for example \cite[I.5.3]{SerreCohGal}). This group is a $k$-form of $H_\xi$ and has $G$ as a normal $k$-subgroup since the action of $Z$ on $G$ by conjugation
is trivial and hence so is the twisting. Moreover, the quotient $\asd{}{\z}{}{\xi}{H}/G$ is clearly $k$-isomorphic to
$F$ since the action of the whole group $G$ on $H_\xi$ by conjugation gets trivialized when one passes to the
quotient $H_\xi/G=F$. We have then a new extension
\[1 \to G \to \asd{}{\z}{}{\xi}{H} \to F \to 1,\]
which actually induces $\kappa$ as well. Indeed, since $H_\xi(k_s)=\asd{}{\z}{}{\xi}{H}(k_s)$, we see that the outer action $\asd{}{\z}{}{}{\kappa}$ of $F$ induced by $\asd{}{\z}{}{\xi}{H}$ coincides trivially with $\kappa$ over $F(k_s)\subset F_\Gamma$, while over $\gamma_F(\Gamma)\subset F_\Gamma$ it also coincides with $\kappa$ merely by the definition of twisting: we have just modified the action of $\gamma_F(\Gamma)$ by \emph{inner} automorphisms.

The extension $H_\xi$ represents then an element $\alpha\cdot\xi$ in $\ext(F,G,\kappa)$. Note that this notation is not abusive, since the choice of another cocycle $\z'$ would give a twisted group $\asd{}{\z'}{}{\xi}{H}$ which would be isomorphic to $\asd{}{\z}{}{\xi}{H}$ and one can easily see that the isomorphism will reduce to identity on $G$ and on $F$, giving thus an isomorphism of extensions. One can then verify that this construction defines an action of $H^1(k,Z)$ on $\ext(F,G,\kappa)$.

\subsection{Comparing $\ext(F,G,\kappa)$ and $H^2(F,G,\kappa)$}\label{section comparaison}
Given an outer action $\kappa:F_\Gamma\to\sout(G_{k_s})$ of $F$ on $G$ as above, its restriction to the subgroup $\gamma_F(\Gamma)\subset F_\Gamma$ is $\kappa_G$ by definition. We have then the natural restriction map
\[H^2(F,G,\kappa)=H^2(F_\Gamma,G(k_s),\kappa)\xrightarrow{\res}H^2(\Gamma,G(k_s),\kappa_G)= H^2(k,G),\]
where $H^2(k,G)$ is the classic set defined in \cite[1.5]{Borovoi93}. Note that this is indeed a map since it corresponds, following \cite[1.18]{SpringerH2}, to the relation $(\gamma_F,\id_G)_*^2$ and $\id_G$ is trivially surjective. It is moreover a surjective map since there is a natural inflation map that goes the other way and defines a section. Recall that $H^2(k,G)$ admits a natural neutral element corresponding to the semi-direct product $G_\Gamma=G(k_s)\rtimes\Gamma$. We can thus consider $H^2(k,G)$ as a pointed set by taking this class as the base point.\footnote{Note that there may be other neutral elements in $H^2(k,G)$. However, they define other actions of $\Gamma$ over $G(k_s)$ (or, if one wishes, other $k$-forms of $G$) and hence our base point is uniquely defined.}

It follows from the definition of this relation that a class in $H^2(F,G,\kappa)$ represented by an extension $E$ of $F_\Gamma$ by $G(k_s)$ maps to this base point if and only if it is represented by a 2-cocycle $(\f,\g)\in Z^2(F,G,\kappa)$ such that
\[\g_{(1,\sigma_1),(1,\sigma_2)}=1,\quad\forall\,\sigma_1,\sigma_2\in \Gamma,\]
and such that $\f_{(1,\sigma)}$ corresponds to the natural action of $\sigma$ on $G(k_s)$. This is also equivalent to the existence of a commutative diagram
\[\xymatrix@R=5mm{
1 \ar[r] & G(k_s) \ar[r] \ar@{=}[d] & G_\Gamma \ar[d] \ar[r] & \Gamma \ar[r] \ar[d]^{\gamma_F} \ar@/_1pc/[l]_{\gamma_G} & 1 \\
1 \ar[r] & G(k_s) \ar[r]  & E \ar[r] & F_\Gamma \ar[r] & 1.
}\]

\begin{pro}\label{proposition ext non abelien}
Let $F$ be a finite smooth $k$-group, $G$ be a smooth algebraic $k$-group and let $Z$ be its smooth center. Assume that there is an outer action $\kappa$ of $F$ on $G$. Then the map $\varphi$ defined in section \ref{section H2 et ext} passes to the quotient of $\ext(F,G,\kappa)$ by the action of $H^1(k,Z)$ and defines a bijection
\begin{equation}\label{isomorphisme}
H^1(k,Z)\backslash\ext(F,G,\kappa)\xrightarrow{\sim} \ker[H^2(F,G,\kappa)\xrightarrow{\res} H^2(k,G)],
\end{equation}
where $\ker(\res)$ means the preimage of the base point in $H^2(k,G)$.
\end{pro}

\begin{rem}
This result, as well as Proposition \ref{proposition ext abelien} here below, can be regarded as a simpler version of \cite[Cor.~3.3.16]{DemarcheHochschild} in a much more restrictive context. In particular, like in the cited result, one can show that a class in $\ext(F,G,\kappa)$ is mapped by $\varphi$ to a neutral class in $H^2(F,G,\kappa)$ if and only if it is a twist of a split extension.
\end{rem}

\medskip

\begin{proof}
We start by showing that the bijection is actually well defined.
\paragraph*{The map \eqref{isomorphisme} is well defined:} It is clear by the construction of $\varphi$ that the
composition $\res\circ\varphi$ gives us the trivial extension $G(k_s)\rtimes\Gamma$ for any $\xi\in\ext(F,G,\kappa)$, thus $\varphi$ actually falls into $\ker(\res)$. We must show now that for $\alpha\in H^1(k,Z)$ and $\xi\in\ext(F,G,\kappa)$ we have $\varphi(\alpha\cdot\xi)=\varphi(\xi)$.\\

Let $\z\in Z^1(k,Z)$ represent $\alpha$, let
\[1 \to G \to {H}_\xi \to F \to 1\qquad (\text{resp.  } 1 \to G \to \asd{}{\z}{}{\xi}{H} \to F \to 1),\]
be an extension representing $\xi$ (resp. $\alpha\cdot\xi$) and let
\[1 \to G(k_s) \to E_\xi \to F_\Gamma \to 1\qquad (\text{resp.  } 1 \to G(k_s) \to \asd{}{\z}{}{\xi}{E} \to F_\Gamma \to 1),\]
be the extension, representing $\varphi(\xi)$ (resp. $\varphi(\alpha\cdot\xi)$), obtained by taking the semi-direct product $E_\xi={H}_\xi(k_s)\rtimes\Gamma$ (resp. $\asd{}{\z}{}{\xi}{E}=\asd{}{\z}{}{\xi}{H}(k_s)\rtimes\Gamma$).

Recall that $Z^1(k,Z)$ classifies sections of the extension $Z(k_s)\rtimes\Gamma$. In particular, we may interpret $\z$ as a section of $E_\xi={H}_\xi(k_s)\rtimes\Gamma$ via the natural inclusion. By definition of twisting, the action of $\Gamma$ on $\asd{}{\z}{}{\xi}{H}(k_s)$ is given by conjugation in $E_\xi$ via the section defined by $\z$. In particular, the two semi-direct products are isomorphic \emph{by definition}, and actually there is a commutative diagram
\[\xymatrix@R=5mm{
1 \ar[r] & H_\xi(k_s) \ar[r] \ar@{=}[d] & E_\xi \ar[r] \ar[d]^{\phi}_{\sim} & \Gamma \ar@{=}[d] \ar[r] & 1 \\
1 \ar[r] & \asd{}{\z}{}{\xi}{H}(k_s) \ar[r] & \asd{}{\z}{}{\xi}{E} \ar[r] & \Gamma \ar[r] & 1.
}\]
In particular, $\phi$ induces the identity on the subgroups $G(k_s)$ of $E_\xi$ and $\asd{}{\z}{}{\xi}{E}$. And since the image of the section $\z$ is in $Z(k_s)\rtimes\Gamma$, both sections coincide once we quotient by $G(k_s)$ on both sides, i.e.~$\phi$ induces the identity on $F(k_s)\rtimes\Gamma=F_\Gamma$. This proves that the map \eqref{isomorphisme} is well defined.

\paragraph*{The map \eqref{isomorphisme} is surjective:}
Take a class in $H^2(F,G,\kappa)$, represented by an extension $E$ of $F_\Gamma$ by $G(k_s)$, such that its image in $H^2(k,G)$ is the base point. We get then a commutative diagram
\[\xymatrix@R=5mm{
1 \ar[r] & G(k_s) \ar[r] \ar@{=}[d] & G_\Gamma \ar[d] \ar[r] & \Gamma \ar[r] \ar[d]^{\gamma_F} \ar@/_1pc/[l]_{\gamma_G} & 1 \\
1 \ar[r] & G(k_s) \ar[r]  & E \ar[r] & F_\Gamma \ar[r] & 1.
}\]
Consider now the preimage of $F(k_s)$ (as a subgroup of $F_\Gamma$) in $E$. This gives us an extension
\begin{equation}\label{extension Gamma-equivariante}
1\to G(k_s)\to \bar H\to F(k_s)\to 1,
\end{equation}
that fits into a commutative diagram
\[\xymatrix@R=5mm{
1 \ar[r] & G(k_s) \ar[r] \ar[d] & G_\Gamma \ar[d] \ar[r] & \Gamma \ar[r] \ar@{=}[d] \ar@/_1pc/[l]_{\gamma_G} & 1 \\
1 \ar[r] & \bar H \ar[r]  & E \ar[r] & \Gamma \ar[r] & 1.
}\]
We see then that the lower exact sequence is split. Hence there is a natural action of $\Gamma$ on $\bar H$, by conjugation in $E$, whose restrictions to $G(k_s)$ and $F(k_s)$ coincide with the natural action of $\Gamma$ given by the $k$-group structure of $G$ and $F$ respectively. In other words, the extension \eqref{extension Gamma-equivariante} is $\Gamma$-equivariant.

Moreover, since $F(k_s)$ is finite and smooth, $\bar H$ can be naturally given the structure of a smooth $k_s$-algebraic group: as a $k_s$-variety, it is a finite union of copies of $G_{k_s}$ (one per element of $F(k_s)$) and the morphisms giving the group structure can be easily defined using the $k_s$-automorphisms $\f_{(f,1)}$ for $f\in F(k_s)$. The action of $\Gamma$ over $\bar H$ is then seen to define semialgebraic automorphisms: this is evident for the copy of $G_{k_s}$ containing the neutral element (since the action of $\Gamma$ over $G$ is by semialgebraic automorphisms) and it can be deduced for the other components from the $k_s$-group structure of $\bar H$ by translation. We conclude by \cite[1.15]{FSS} that $\bar H$ descends into a smooth algebraic $k$-group representing an element in $\ext(F,G,\kappa)$. The fact that this element is a preimage of the given class in $H^2(F,G,\kappa)$ is obvious.

\paragraph*{The map \eqref{isomorphisme} is injective:} Let $\xi_1,\xi_2\in\ext(F,G,\kappa)$ be elements such that
$\varphi(\xi_1)=\varphi(\xi_2)$. Since we have $\varphi(\xi_i)\in\ker(\res)$, this means that we have the following
commutative diagram with exact rows:
\begin{equation}\label{diagramme prisme triangulaire}
\xymatrix@R=5mm{
1 \ar[r] & G(k_s) \ar[rr] \ar@{=}[dd] && E_{\xi_1} \ar[dd]^<<<<<<<{\phi}_<<<<<<<{\sim} \ar[rr] && F_\Gamma \ar[r] \ar@{=}[dd] & 1 \\
& \quad 1 \ar[r] & G(k_s) \ar@{=}[ul] \ar@{=}[dl] \ar'[r][rr] && G_\Gamma \ar'[r][rr] \ar[ul]_{\iota_1} \ar[dl]^{\iota_2} && \Gamma \ar[ul]_{\gamma_F} \ar[dl]^{\gamma_F} \ar[r] \ar@/_1pc/[ll]_>>>>>{\gamma_G} & 1 \\
1 \ar[r] & G(k_s) \ar[rr] && E_{\xi_2} \ar[rr] && F_\Gamma \ar[r] & 1
}
\end{equation}
where, for $i=1,2$, $E_{\xi_i}$ is the semi-direct product $H_{\xi_i}(k_s)\rtimes\Gamma$. Let $(\f^i,\g^i)\in Z^2(F,G,\kappa)$ be a cocycle representing $\varphi(\xi_i)$. The diagram tells us that we may choose them in such a way that we have, for $\sigma,\tau\in\Gamma$,
\begin{align}
\f^1_{(1,\sigma)}&=\f^2_{(1,\sigma)}=\mathrm{int}(\gamma_G(\sigma)) ;\label{egalite des f sur Gamma}\\
\g^1_{(1,\sigma),(1,\tau)}&=\g^2_{(1,\sigma),(1,\tau)}=1, \label{nullite de g sur Gamma}
\end{align}
where, in order to lighten notation, we have identified $E_{\xi_1}$ and $E_{\xi_2}$ via $\phi$ and also $\gamma_G(\sigma)$ with its image in $E_{\xi_i}$ via $\iota_i$. Now, since these two 2-cocycles represent the same class in $H^2(F,G,\kappa)$, there exists a continuous map $\z:F_\Gamma\to G(k_s)$ such that, for all $x,y\in F_\Gamma$
\begin{align}
\f^2_{x}&=\mathrm{int}(\z_{x})\circ\f^1_{x},\label{relation entre les f}\\
\g^2_{x,y}&=\z_{xy}\g^1_{x,y}\f^1_{x}(\z_{y})^{-1}\z_{x}^{-1}.\label{relation entre les g}
\end{align}
Let us recall how to obtain such a map $\z$. The cocycles $(\f^i,\g^i)$ are actually dependant of the choice of respective sections $s_i:F_\Gamma\to E_{\xi_i}$, both compatible with $\gamma_G$ in order to satisfy \eqref{egalite des f sur Gamma} and \eqref{nullite de g sur Gamma}. Writing $(g,x):=g\cdot s_i(x)$ for $g\in G(k_s)$ and $x\in F_\Gamma$, the group $E_{\xi_i}$ can be seen as the set $G(k_s)\times F_\Gamma$ with the group law given by
\[(g_1,x)\cdot(g_2,y)=(g_1\f^i_{x}(g_2)\g^i_{x,y},xy),\]
for $g_1,g_2\in G(k_s)$, $x,y\in F_\Gamma$. In order to respect the commutativity of diagram \eqref{diagramme prisme triangulaire}, the isomorphism $\phi$ must then satisfy, for $g\in G(k_s)$, $x=(f,\sigma)\in F_\Gamma$,
\begin{align*}
\phi(g,(1,1))&=(g,(1,1)),\\
\phi(1,(f,\sigma))&=(\z_{(f,\sigma)},(f,\sigma)),
\end{align*}
with $\z_{(f,\sigma)}\in G(k_s)$. This is precisely our map $\z:F_\Gamma\to G(k_s)$. Note that in particular we have
\begin{equation}\label{relation entre phi et z}
\phi(\iota_1(\gamma_G(\sigma)))=\z_{(1,\sigma)}\cdot \iota_2(\gamma_G(\sigma)),\quad\forall\,\sigma\in\Gamma.
\end{equation}

Now, applying the different equalities on our 2-cocycles above, we get that:
\begin{itemize}
\item The restriction of $\z$ to $\gamma_F(\Gamma)$ takes values in $Z(k_s)$, since \eqref{egalite des f sur Gamma} and \eqref{relation entre les f} imply $\mathrm{int}(\z_{(1,\sigma)})=\mathrm{id}_G$ and we know that the $k_s$-points of $Z$ are those of the center of $G$.
\item The restriction of $\z$ to $\gamma_F(\Gamma)$ is a 1-cocycle, since \eqref{nullite de g sur Gamma} and \eqref{relation entre les g} imply
\[\z_{(1,\sigma)(1,\tau)}\f^1_{(1,\sigma)}(\z_{(1,\tau)})^{-1}\z_{(1,\sigma)}^{-1}=1,\quad\text{i.e.}
\quad \z_{\sigma\tau}=\z_{\sigma}\asd{\sigma}{}{}{\tau}{\z},\quad \forall \sigma,\tau\in\Gamma.\]
\end{itemize}
Otherwise stated, we get that the restriction of $\z$ to $\gamma_F(\Gamma)\cong\Gamma$ is in $Z^1(k,Z)$. We claim that this cocycle will do, i.e. if we denote
\[1\to G\to H_{\xi_i}\to F\to 1,\quad i=1,2,\]
the extensions representing respectively $\xi_1$ and $\xi_2$, we have $H_{\xi_2}\cong\asd{}{\z}{}{\xi_1}{H}$.\\

In order to prove this last assertion, it will suffice to prove that $\Gamma$ acts in the same way on $\asd{}{\z}{}{\xi_1}{H}(k_s)$ and $H_{\xi_2}(k_s)$. Notice that these two groups are identified via the isomorphism $\phi$ of diagram \eqref{diagramme prisme triangulaire} as the respective preimages of $F(k_s)\subset F_\Gamma$ in $E_{\xi_1}$ and $E_{\xi_2}$ (recall that $\asd{}{\z}{}{\xi_1}{H}(k_s)=H_{\xi_1}(k_s)$ as groups). Moreover, the action of $\Gamma$ on $E_{\xi_i}$ is given by conjugation when one considers $\Gamma$ as a subgroup via $\iota_i\circ \gamma_G$ in diagram \eqref{diagramme prisme triangulaire}. Thus, $\sigma\in\Gamma$ acts on $H_{\xi_i}(k_s)$ via $\mathrm{int}(\iota_i(\gamma_G(\sigma)))$, hence on $\asd{}{\z}{}{\xi_1}{H}(k_s)$ via $\mathrm{int}(\z_\sigma)\circ\mathrm{int}(\iota_1(\gamma_G(\sigma)))$ by the definition of twisting. Recalling then equation \eqref{relation entre phi et z}, we see immediately that these actions are the same on $H_{\xi_2}(k_s)$ and $\asd{}{\z}{}{\xi_1}{H}(k_s)$, which concludes the proof.
\end{proof}

If one is interested in the stabilizers of the action of $H^1(k,Z)$ on $\ext(F,G,\kappa)$, then the first thing one should note is that they are trivial for obvious reasons when $k=k_s$. Now, since $k$-forms are always classified by the $H^1$ of the automorphism group, we must then try to see what this group looks like.

\begin{pro}\label{prop aut E}
Let $E$ be an extension of a finite smooth $k$-group $F$ of order $n$ by a smooth algebraic $k$-group $G$ and let $Z$ be the smooth center of $G$. Then the automorphism group of the extension (i.e.~automorphisms of $E$ that induce the identity on $F$ and $G$) is isomorphic to the $k$-group $Z^1(F,Z)$ of Hochschild 1-cocycles.
\end{pro}

Recall (cf.~for instance \cite[II.3.1]{DemazureGabriel}) that the $k$-points of $Z^1(F,Z)$ are $k$-morphisms $c:F\to Z$ such that, if we denote by $a:F\times Z\to Z$ the $k$-action induced by $E$, by $\rho$ the composition
\[F\xrightarrow{\id\times c} F\times_k Z \xrightarrow{a} Z,\]
and by $m_F,m_Z$ the respective multiplication morphisms, then the following diagram commutes:
\[\xymatrix{
F\times_k F \ar[d]_{m_F} \ar[r]^{c\times\rho} & Z\times_k Z \ar[d]^{m_Z}\\
F \ar[r]^c & Z.
}\]

\begin{proof}
This a straightforward calculation. Let $\varphi:E\to E$ be an isomorphism inducing the identity on $G$ and $F$. Since all groups are smooth, the composite morphism
\[E\xrightarrow{\Delta} E\times_k E \xrightarrow{\id\times(\iota_E\circ\varphi)} E\times_k E\xrightarrow{m_E} E,\]
where $\Delta$ is the diagonal morphism and $\iota_E,m_E$ are the inverse and multiplication in $E$, is easily seen to fall into the smooth center $Z$ of $G\subset E$. Moreover, it also passes to the quotient by $G$ and hence defines a $k$-morphism $F\to Z$ which corresponds to an element of $Z^1(F,Z)$.

On the other hand, a morphism $c:F\to Z$ in $Z^1(F,Z)$ induces an isomorphism $\varphi$ of $E$ fixing both $F$ and $G$ by considering the composition
\[E\xrightarrow{\id\times(c\circ\pi)} E\times Z\xrightarrow{m_E} E,\]
where $\pi$ denotes the projection $E\to F$. These two constructions are the inverse of each other.
\end{proof}

Having proved this, we see that every element in $\ext(F,G,\kappa)$ has the same automorphism group and hence $H^1(k,Z^1(F,Z))$ acts naturally on $\ext(F,G,\kappa)$. Then, by construction, we get the following.

\begin{cor}\label{cor aut E}
The action of $H^1(k,Z)$ on $\ext(F,G,\kappa)$ is via its image on $H^1(k,Z^1(F,Z))$ through the coboundary morphism
\[Z\to Z^1_0(F,Z):z\mapsto (f\mapsto z\asd{f}{}{-1}{}{z}).\]
In particular, the stabilizers are \emph{all the same} and include the image of $H^1(k,Z^F)$ in $H^1(k,Z)$, where $Z^F$ denotes the $k$-subgroup of $F$-invariant elements.\qed
\end{cor}

\begin{rem}
Denote by $B^1(F,Z)$ the image of the coboundary morphism. The stabilizers could then be bigger if the map $H^1(k,B^1(F,Z))\to H^1(k,Z^1(F,Z))$ is not injective, which amounts to the non-surjectivity of the map $Z^1(F,Z)(k)\to H^1_0(\bar F,\bar Z)^\Gamma$, where $H^1_0$ denotes the Hochschild cohomology. This does not seem to be evident to verify however.
\end{rem}

\subsection{The case of a commutative group}\label{section gps commutatifs}
In the particular case of a smooth commutative group $A$, we have $Z=A$ and the outer action of $F$ on $A$ becomes an action by group automorphisms by Proposition \ref{proposition k actions}. It is well known that in this setting the sets $\ext(F,A,\kappa)$, $H^2(F,A,\kappa)$ and $H^i(k,A)$ have a natural group structure (see for example \cite[XVII, App.~I]{SGA3} for the first one, \cite[I.2, II.1]{SerreCohGal} for the second and third ones). One can easily see then that the map
\begin{align*}
H^1(k,A)&\xrightarrow{\psi}\ext(F,A,\kappa)\\
\alpha&\mapsto\alpha\cdot 0,
\end{align*}
is a group homomorphism, where $0$ represents the trivial element in $\ext(F,A,\kappa)$ (i.e. the class of the semidirect product). In the same fashion, the maps
\begin{align*}
\ext(F,A,\kappa)&\xrightarrow{\varphi} H^2(F,A,\kappa),\\
H^2(F,A,\kappa)&\xrightarrow{\res} H^2(k,A),
\end{align*}
defined as above, are also group homomorphisms. For the former, one only needs to remark that $H^2(F,A,\kappa)$ is described by extensions of $F_\Gamma$ by $A(k_s)$ and the way of mutiplying extensions in both cases is the same. Proposition \ref{proposition ext non abelien} then becomes:

\begin{pro}\label{proposition ext abelien}
Let $A$, $F$, be smooth algebraic $k$-groups such that $A$ is commutative, $F$ is finite of order $n$ and $F$ acts on $A$ by group automorphisms (via $\kappa$, as above). Then the sequence
\[H^1(k,A)\xrightarrow{\psi}\ext(F,A,\kappa)\xrightarrow{\varphi} H^2(F,A,\kappa)\xrightarrow{\res} H^2(k,A),\]
is exact. In particular, $\ext(F,A,\kappa)$ is a torsion group and, if $H^1(k,A)$ is $d$-torsion,
then $\ext(F,A,\kappa)$ is an $nd$-torsion group.
\end{pro}

\begin{proof}
We only have to prove the last two statements. The classic restriction-corestriction argument tells us that the image of $\varphi$ is killed by $n$ since $\gamma_F(\Gamma)$ is of index $n$ in $F_\Gamma$. On the other side, it is well known that $H^1(k,A)$ is torsion simply because $\Gamma$ is profinite (cf.~\cite[I.2.2, Cor.~3]{SerreCohGal}). The last assertion is then evident.
\end{proof}

\subsection{Description of $\ext(F,G,\kappa)$ via the smooth center of $G$}\label{section H2L et H2Z pour F-liens}
It is a known fact that the $H^2$ set of a given kernel is a principal homogeneous space of the $H^2$ group of its center (cf.~for instance \cite[1.17]{SpringerH2}). Considering that this has a clear translation into the language of group extensions, one would be tempted to see if there is such a relation between the set $\ext(F,G,\kappa)$ and the corresponding group $\ext(F,Z,\kappa)$, where we abusively still denote by $\kappa$ the kernel induced on the smooth center $Z$ of $G$.

In order to do so, let us recall first how to obtain an $F$-kernel (and hence an action) on the smooth center of a group with an outer action. Let $G$ be a smooth algebraic $k$-group and let $Z$ be its smooth center. Assume that there is an outer action $\kappa$ of $F$ on $G$. Then it suffices to compose $\kappa$ with the natural homomorphism $\sout(G_{k_s})\to\sout(Z_{k_s})$ induced by the same homomorphism at the level of $\saut$. Recalling then that $\sout(Z_{k_s})=\saut(Z_{k_s})$ because $Z$ is commutative, and that $\kappa$ is a splitting of \eqref{equation suite exacte avec SOut G} when restricted to $\gamma_F(\Gamma)$, we get the desired action by Proposition \ref{proposition k actions}.

As a particular case of the nonabelian group cohomology theory, we get then the following result, cf.~\cite[1.17]{SpringerH2}.

\begin{pro}\label{proposition action de H2Z sur H2G F-liens}
Let $F$ be a finite algebraic smooth $k$-group and let $G$ be a smooth algebraic $k$-group. Assume that there is an outer action $\kappa$ of $F$ on $G$. Let $Z$ be the smooth center of $G$ and denote also by $\kappa$ the algebraic action induced on $Z$. Then $H^2(F,Z,\kappa)$ acts freely and transitively on $H^2(F,G,\kappa)$.\qed
\end{pro}

Let us now define an action of $\ext(F,Z,\kappa)$ on $\ext(F,G,\kappa)$ by following the construction in abstract group theory. Consider $\xi\in\ext(F,G,\kappa)$ and $\zeta\in\ext(F,Z,\kappa)$. Take extensions
\[1\to G\to H_\xi\to F\to 1\qquad\text{and}\qquad 1\to Z\to H_{\zeta}\to F\to 1,\]
representing $\xi$ and $\zeta$ respectively and consider their direct product
\[1\to G\times_k Z\to H_\xi\times_k H_{\zeta}\to F\times_k F\to 1.\]
Consider now the diagonal subgroup $F$ of $F\times_k F$ and take preimages in order to get 
\begin{equation}\label{extension action de Z sur G}
1\to G\times_k Z\to H \to F\to 1.
\end{equation}
Consider now the multiplication $k$-morphism $G\times_k Z\to G$. This is clearly $H$-equivariant and hence the kernel
of such a morphism is normal in $H$. One may thus take quotients by this kernel in order to get the extension
\[1\to G\to H_{\xi'}\to F\to 1,\]
representing a class $\xi'\in\ext(F,G,\kappa)$. We set then $\zeta\cdot\xi:=\xi'$. This definition is clearly
independent of the choice of the extensions representing $\xi$ and $\zeta$.\\

Fix now a class $\xi\in\ext(F,G,\kappa)$ and define the following maps:
\[\begin{array}{rccccccl}
\psi_\xi : & H^1(k,Z) & \to & \ext(F,G,\kappa) & : & \alpha & \mapsto & \alpha\cdot\xi,\\
\theta_\xi : & \ext(F,Z,\kappa) & \to & \ext(F,G,\kappa) & : & \zeta & \mapsto &\zeta\cdot\xi,\\
\theta_{\varphi(\xi)} : & H^2(F,Z,\kappa) & \to & H^2(F,G,\kappa) & : & \eta & \mapsto & \eta\cdot\varphi(\xi),\\
\theta : & H^2(k,Z) & \to & H^2(k,G) & : & \eta & \mapsto & \eta\cdot\eta_0,
\end{array}\]
where $\eta_0$ denotes the base point of $H^2(k,G)$ as defined at the beginning of section \ref{section comparaison}. The third and fourth maps are immediately seen to be bijective by Proposition \ref{proposition action de H2Z sur H2G F-liens} and \cite[1.17]{SpringerH2}.

\begin{pro}\label{proposition compatibilite ext}
Under the hypotheses of Proposition \ref{proposition action de H2Z sur H2G F-liens}, there is a commutative diagram with exact rows
\[\xymatrix@R=5mm{
\ker(\psi)\ar[r] \ar@{=}[d] & H^1(k,Z) \ar[r]^>>>>>{\psi} \ar@{=}[d] & \ext(F,Z,\kappa) \ar[d]^{\theta_\xi} \ar[r]^{\varphi} & H^2(F,Z,\kappa) \ar[d]^{\theta_{\varphi(\xi)}}_\sim \ar[r]^>>>>>{\res} & H^2(k,Z)
\ar[d]^{\theta}_\sim \\
\ker(\psi) \ar[r] & H^1(k,Z) \ar[r]^>>>>>{\psi_\xi} & \ext(F,G,\kappa) \ar[r]^{\varphi} &
H^2(F,G,\kappa) \ar[r]^>>>>>{\res} & H^2(k,G),
}\]
where the base point in $H^2(F,G,\kappa)$ is $\varphi(\xi)$. In particular, the action of $\ext(F,Z,\kappa)$ on $\ext(F,G,\kappa)$ is simply transitive.
\end{pro}

\begin{proof}
The exactness of both rows follows from Propositions \ref{proposition ext non abelien} and \ref{proposition ext abelien} and Corollary \ref{cor aut E}. The last one is needed to prove that $\ker(\psi)$ is indeed the kernel on the second row.

Commutativity in the square on the right hand side is trivial when one looks at it at the level of cocycles. The same is true for the two squares on the left hand side, by Corollary \ref{cor aut E} for the first and by functoriality of the twisting procedure for the second. We must prove then commutativity in the middle square, which is also done at the level of cocycles. Indeed, it suffices to follow the definition of the action of $\ext(F,Z,\kappa)$ over $\ext(F,G,\kappa)$ defined above. Let $(\f,\g)$ be a cocycle representing $\varphi(\xi)$. Then $\f$ restricts to $Z(k_s)$ as the natural action of $F_\Gamma$. In particular, we may denote by $(\f|_Z,\z)$ a cocycle representing $\varphi(\zeta)$ for $\zeta\in\ext(F,Z,\kappa)$. It is then easy to see that the image by $\varphi$ of the extension \eqref{extension action de Z sur G} is represented by $(\f\times\f|_Z,\g\times\z)$ and hence the cocycle representing $\varphi(\zeta\cdot\xi)$ is $(\f,\z\g)$ since $\f_x$ is compatible with the multiplication morphism for $x\in F_\Gamma$. This is precisely the cocycle representing $\theta_{\varphi(\xi)}(\varphi(\zeta))=\varphi(\zeta)\cdot\varphi(\xi)$ by \cite[1.17]{SpringerH2}, which proves commutativity.

Finally, the simple transitivity of the action of $\ext(F,Z,\kappa)$ on $\ext(F,G,\kappa)$ is easily proved by diagram chasing.
\end{proof}

\section{Reduction of extensions}\label{section reduction}
In this section, we follow the ``d\'evissage'' ideas in \cite[\S3]{SpringerH2}, in which Springer proves the following theorem in the context of $k$-kernels.

\begin{thm}\cite[Thm.~3.4]{SpringerH2}
Let $k$ be a perfect field, let $\bar G$ be a (smooth) algebraic $k_s$-group and let $\kappa$ be a $k$-kernel in
$\bar G$. Then for every $\eta\in H^2(k,\bar G,\kappa)$ there exists a finite nilpotent subgroup $\bar H$ of $\bar G$
and a $k$-kernel $\lambda$ in $\bar H$, compatible with $\kappa$, such that
$\eta\in\iota_*^2(H^2(k,\bar H,\lambda))$, where $\iota:\bar H\to \bar G$ denotes the inclusion.
\end{thm}

The notation $\eta\in\iota_*^2(H^2(k,\bar H,\lambda))$ means that there exists a class $\xi\in H^2(k,\bar H,\lambda)$ that is related to $\eta$ via the relation $\iota^2_*$ (cf.~\cite[1.18]{SpringerH2}). This is equivalent to the existence of a commutative diagram
\[\xymatrix@R=5mm{
1 \ar[r] & \bar H(k_s) \ar[r] \ar[d]^{\iota} & E_\xi \ar[d] \ar[r] & \Gamma \ar[r] \ar@{=}[d] & 1 \\
1 \ar[r] & \bar G(k_s) \ar[r]  & E_\eta \ar[r] & \Gamma \ar[r] & 1,
}\]
where $\Gamma$ denotes $\gal(k_s/k)$ as always, $E_\xi$ represents $\xi$ and $E_\eta$ represents $\eta$.

\begin{rem}
Springer's theorem asserts moreover that $\bar H$ is defined over $k$. This assertion however does not make much sense unless the group $\bar G$ itself is defined over $k$ (a finite group can always be given the structure of a constant group over \emph{any} field). We decided then to take this assertion as a typo, since there is no mention of it in the proof (Springer's Proposition 3.1 and Lemmas 3.2, 3.3, which are used in his proof, do not have this assertion). One could wonder if Springer actually had a proof of the existence of finite $k$-groups factoring extensions for an arbitrary algebraic $k$-group. All the more since we show here below that the same techniques can actually give such a result.
\end{rem}

We restrict from now on to a \emph{perfect} field $k$.

\begin{thm}\label{theoreme reduction ext}
Let $k$ be a perfect field of characteristic $p\geq 0$. Let $F$ be a smooth finite $k$-group of order $n$, and $G$ an
arbitrary smooth $k$-group. Then, given an extension
\begin{equation}\label{extension theoreme}
1\to G \to H\to F\to 1,\tag{$\star$}
\end{equation}
there exists a finite smooth $k$-subgroup $S$ of $G$ and a commutative diagram with exact rows
\[\xymatrix@R=5mm{
1 \ar[r] & S \ar[d] \ar[r] & H' \ar[d] \ar[r] & F \ar@{=}[d] \ar[r] & 1 \\
1 \ar[r] & G \ar[r] & H \ar[r] & F \ar[r] & 1.
}\]
Moreover, if $G$ is linear, let $T$ be a maximal torus in $G$, $W$ be the Weyl group of $G$ (that is, the finite group of connected components of the normalizer of $T$) and $K/k$ be a separable algebraic extension splitting $T$. Denote by $r$ the rank of $T$, $w$ the order of $W$ and $d$ the degree of $K/k$. Assume that either $nw$ is prime to $p$ or that $G^\circ$ is reductive. Then one can take $S$ to be contained in an extension of $W$ by the $ndw$-torsion subgroup of $T$, hence of order dividing $(nd)^rw^{r+1}$.
\end{thm}

\begin{rem}
The first part of this result has been recently proved by Brion (cf.~\cite[Thm.~1]{Brion}) for any field $k$ and with no smoothness assumption on $G$ or $F$. It had already been claimed to be true for any perfect field $k$ more than 50 years ago by Borel and Serre (cf.~\cite[Lem.~5.11, footnote on p.~152]{BorelSerre}), although they only gave the proof for linear $G$ and $k=k_s$ of characteristic zero. Platonov gave shortly after a proof of this fact for linear groups over a perfect field (cf.~\cite[Lem.~4.14]{Platonov}). The second part has been already proved by the author in the particular case in which $G$ is itself a torus and $nd$ is prime to $p$. Other particular cases had been treated for example in \cite[Prop.~7]{Vinberg} and \cite[Lem.~5.2]{Reichstein et cie}.
\end{rem}

Before we get started, let us prove a technical lemma that will be used in the proof.

\begin{lem}\label{lemme technique}
Let $k$ be a field. Let $F$ be a smooth finite $k$-group and $G$ an arbitrary smooth $k$-group. Let $H$ be an extension of $F$ by $G$, $\kappa$ the outer action of $F$ on $G$ induced by this extension and $\xi\in\ext(F,G,\kappa)$ the corresponding class. Let $(\f,\g)\in Z^2(F,G,\kappa)$ be a cocycle representing $\varphi(\xi)\in H^2(F,G,\kappa)$ (cf.~section \ref{section H2 et ext}).

Assume that the restriction of $(\f,\g)$ to $\gamma_F(\Gamma)$ gives the trivial cocycle $(\f_G,1)$ associated to the $k$-group $G$ (cf.~section \ref{section k-liens}). Assume moreover that there exists a smooth $k_s$-subgroup $\bar M$ of $G_{k_s}$, invariant by $\f_{(f,\sigma)}$ for all $(f,\sigma)\in F_\Gamma$, such that the image of $\g$ is contained in $\bar M(k_s)$.

Then there exists a smooth $k$-subgroup $M$ of $G$ inducing the inclusion of $\bar M$ and a commutative diagram with exact rows
\[\xymatrix@R=5mm{
1 \ar[r] & M \ar[d] \ar[r] & H' \ar[d] \ar[r] & F \ar@{=}[d] \ar[r] & 1 \\
1 \ar[r] & G \ar[r] & H \ar[r] & F \ar[r] & 1.
}\]
\end{lem}

\begin{proof}
Since the restriction of $\f$ to $\gamma_F(\Gamma)$ corresponds to the natural action of $\Gamma$ on $G(k_s)$ by hypothesis and since $\bar M(k_s)$ is stable by this action, we get that $\bar M$ comes actually from a smooth $k$-subgroup $M$ of $G$. The cocycle $(\f,\g)$ can then be seen as an element of $Z^2(F,M,\kappa')$ where $\kappa'$ is the $F$-kernel (and outer action) induced by the restriction of $\f$ to $M$. By ``functoriality'' of the $H^2$ sets (cf.~\cite[1.18]{SpringerH2}) with respect to the inclusions $M(k_s)\to G(k_s)$ and $\gamma_F(\Gamma)\to F_\Gamma$, this implies the existence of a commutative diagram with exact rows
\[\xymatrix@R=3mm{
& 1 \ar[r] & M(k_s) \ar[rr] \ar@{=}[dl] \ar'[d][dd] && M_\Gamma \ar'[d][dd] \ar[rr] \ar[dl] && \Gamma \ar[r] \ar@{=}'[d][dd] \ar[dl]_{\gamma_F} \ar@/_1pc/[ll]_{\gamma_M}  & 1 \\
1 \ar[r] & M(k_s) \ar[rr] \ar[dd] && E'\ar[dd] \ar[rr] && F_\Gamma \ar@{=}[dd] \ar[r] & 1 \\
& \quad 1 \ar[r] & G(k_s) \ar'[r][rr] \ar@{=}[dl] && G_\Gamma \ar'[r][rr] \ar[dl] && \Gamma \ar[r] \ar[dl]^<<<<<{\gamma_F} \ar@/_1pc/[ll]_>>>>>>{\gamma_G} & 1 \\
1 \ar[r] & G(k_s) \ar[rr] && E \ar[rr] && F_\Gamma \ar[r] & 1,
}\]
where $E$ is an extension representing $\varphi(\xi)$ and $E'$ represents the class in $H^2(F,M,\kappa')$ given by the cocycle $(\f,\g)$, which ``maps'' to $\varphi(\xi)$.

Consider the preimages of $F(k_s)\subset F_\Gamma$ in the front rows of the diagram. We get a subgroup $\bar H'$ of $H(k_s)$ that fits into the following commutative diagram
\[\xymatrix@R=3mm{
& 1 \ar[r] & M(k_s) \ar[rr] \ar[dl] \ar'[d][dd] && M_\Gamma \ar'[d][dd] \ar[rr] \ar[dl] && \Gamma \ar[r] \ar@{=}'[d][dd] \ar@{=}[dl] \ar@/_1pc/[ll]_{\gamma_M}  & 1 \\
1 \ar[r] & \bar H'(k_s) \ar[rr] \ar[dd] && E'\ar[dd] \ar[rr] && \Gamma \ar@{=}[dd] \ar[r] & 1 \\
& \quad 1 \ar[r] & G(k_s) \ar'[r][rr] \ar[dl] && G_\Gamma \ar'[r][rr] \ar[dl] && \Gamma \ar[r] \ar@{=}[dl] \ar@/_1pc/[ll]_>>>>>{\gamma_G} & 1 \\
1 \ar[r] & H(k_s) \ar[rr] && E \ar[rr] && \Gamma \ar[r] & 1,
}\]
where we see $\bar H'$ as a smooth $k_s$-subgroup of $H_{k_s}$. This can be done since its connected component $\bar M=M_{k_s}$ already is a $k_s$-subgroup and $\bar H'$ is just a finite set of translates of $M_{k_s}$ (one per element of $F(k_s)$).

The diagram tells us that there are natural compatible splittings of the exact sequences on the front. The splitting on the lower part of the diagram gives moreover the natural action of $\Gamma$ on $H(k_s)$ (recall that $E=H(k_s)\rtimes\Gamma$ by definition). We see then that $\bar H'(k_s)$ is stable by this action. We deduce then that $\bar H'$ comes in fact from a $k$-subgroup $H'$ of $H$. This $k$-group is clearly an extension of $F$ by $M$ and fits into a commutative diagram as the one in the statement of the lemma.
\end{proof}

\begin{proof}[Proof of Theorem \ref{theoreme reduction ext}]
The extension \eqref{extension theoreme} induces an outer action $\kappa$ of $F$ on $G$ and corresponds to a class $\xi\in\ext(F,G,\kappa)$. The proof follows the reasoning of Springer and goes by successive generalizations.

\paragraph*{Step 1: the case of tori and abelian varieties.}
Assume that $G$ is either a torus or an abelian variety. Then $\kappa$ defines an action by group automorphisms of $F$ on $G$ by Proposition \ref{proposition k actions}. The result is then a consequence of Proposition \ref{proposition ext abelien}. Indeed, since $\ext(F,G,\kappa)$ is a torsion group, then there exists some $m\in\mathbb{N}$ such that $m\cdot\xi=0$. Consider then the exact sequence
\[1\to G[m]\to G\xrightarrow{m} G\to 1,\]
where $G[m]$ denotes the $m$-torsion of $G$. Using diagram \eqref{diagramme action de groupes} one can see that this action restricts to an action of $F$ on $G[m]$, abusively still denoted by $\kappa$. We get then an exact sequence (cf.~\cite[XVII, App.~I, Prop.~2.1]{SGA3})
\[\ext(F,G[m],\kappa)\to\ext(F,G,\kappa)\xrightarrow{m}\ext(F,G,\kappa).\]
We see thus that $\xi$ comes from $\ext(F,G[m],\kappa)$, which gives us a commutative diagram
\[\xymatrix@R=5mm{
1 \ar[r] & G[m] \ar[d] \ar[r] & H' \ar[d] \ar[r] & F \ar@{=}[d] \ar[r] & 1 \\
1 \ar[r] & G \ar[r] & H \ar[r] & F \ar[r] & 1.
}\]
Recall now that $S:=G[m]$ is finite for $G$ a torus or an abelian variety. Moreover, in the case of a torus, we know that $H^1(k,G)$ is $d$-torsion by the classic restriction-corestriction argument and Hilbert's theorem 90. We deduce again from Proposition \ref{proposition ext abelien} that we may take $m=nd$ and hence the order of $S$ is $(nd)^r$.

Finally, consider the smooth $k$-subgroup $H'_\red$ of $H'$ corresponding to the reduced subscheme of $H'$ (which is a subgroup since $k$ is perfect). This is easily seen to be an extension of $F$ by $S_\red$, which is a smooth $k$-subgroup of $S$ and hence of order dividing $(nd)^r$ in the case where $G$ is a torus.

\paragraph*{Step 2 : the connected unipotent case.}
Since $k$ is perfect, we know that $G$ admits a characteristic decomposition series in which every quotient is
isomorphic to $\bb{G}_\mathrm{a}$ (cf.~\cite[XVII, Cor.~4.1.3]{SGA3}).

Let us treat first the case where $G\cong\bb{G}_\mathrm{a}$. In this case, Proposition \ref{proposition ext abelien} tells us that $\ext(F,G,\kappa)$ is isomorphic to $H^2(F,G,\kappa)$ since $H^i(k,G)=0$ for $i=1,2$ (cf.~\cite[III.2.1, Prop.~6]{SerreCohGal}). Assume that $n$ is prime to $p$ and consider the Hochschild-Serre spectral sequence in group cohomology
\[H^p(\Gamma,H^q(F(k_s),G(k_s)))\Rightarrow H^{p+q}(F_\Gamma,G(k_s))=H^{p+q}(F,G,\kappa).\]
Then $H^q(F(k_s),G(k_s))=0$ for every $q>0$ since multiplication by $n$ in $G(k_s)$ is an isomorphism and $H^q(F(k_s),G(k_s))$ is $n$-torsion. Thus $H^2(F,G,\kappa)$ is isomorphic to $H^2(\Gamma,H^0(F(k_s),G(k_s)))$. Since the automorphism group of $G=\bb{G}_{\mathrm{a}}$ is $\gm$, we see by Proposition \ref{proposition k actions} that the action of $F(k_s)$ on $G(k_s)=k_s$ must be by homotheties and hence $H^0(F(k_s),G(k_s))$ is either $G(k_s)$ or $0$ and thus $H^2(\Gamma,H^0(F(k_s),G(k_s)))$ is also trivial (since $H^2(\Gamma,G(k_s))=H^2(k,G)=0$). We deduce that in this case the group $\ext(F,G,\kappa)$ is trivial and hence \eqref{extension theoreme} is a semi-direct product and one can take $H'\cong F$.

Assume now that $k$ has positive characteristic and make no assumptions on $n$. Then, if $(\f,\g)$ is a cocycle representing $\varphi(\xi)\in H^2(F,G,\kappa)$, we know that $\g:F_\Gamma^2\to G(k_s)$ has a finite image by continuity since $F_\Gamma$ is profinite. Moreover, $\f$ is \emph{equal} to $\kappa_G$ over $\gamma_F(\Gamma)$ since $G$ is abelian. Then the $F_\Gamma$-invariant subgroup $\bar S$ of $G(k_s)$ generated by the image of $\g$ is also finite since $G(k_s)$ is $p$-torsion and the action of $F_\Gamma$ over $G(k_s)$ is continuous. Thus $\bar S$ may be considered as a smooth $k_s$-subgroup. Lemma \ref{lemme technique} applies then, giving us a $k$-form $S$ of $\bar S$ and a commutative diagram with exact rows
\[\xymatrix@R=5mm{
1 \ar[r] & S \ar[d] \ar[r] & H' \ar[d] \ar[r] & F \ar@{=}[d] \ar[r] & 1 \\
1 \ar[r] & G \ar[r] & H \ar[r] & F \ar[r] & 1.
}\]

The general case is easily deduced from the previous one by induction. Indeed, let $G/G_1\cong\bb{G}_\mathrm{a}$ be the first quotient obtained from the characteristic decomposition of $G$. We assume that the result is true for $G_1$ and any finite group $F$. Since $G_1$ is characteristic in $G$, it is normal in $H$. Taking then quotients by $G_1$, we get an extension
\[1\to \bb{G}_\mathrm{a}\to H/G_1\to F\to 1,\]
from which we can get a smooth finite $k$-subgroup $S_0$ of $\bb{G}_\mathrm{a}$ and a commutative diagram
\[\xymatrix@R=5mm{
1 \ar[r] & S_0 \ar[d] \ar[r] & H_0 \ar[d] \ar[r] & F \ar@{=}[d] \ar[r] & 1 \\
1 \ar[r] & \bb{G}_\mathrm{a} \ar[r] & H/G_1 \ar[r] & F \ar[r] & 1.
}\]
Consider then the preimage of $H_0$ in $H$: it is an extension of the finite $k$-group $H_0$ by $G_1$ and hence by assumption it admits a finite $k$-subgroup $H'$ extension of $H_0$ by a finite subgroup $S_1$ of $G_1$. This $H'$ is clearly an extension of $F$ by a finite subgroup $S$ of $G$ pulling back \eqref{extension theoreme}.

Assume at last once again that $n$ is prime to $p$. Then it is clear that in each induction step above we may assume the groups $S_0$ and $S_1$ to be trivial and hence the finite group $H'$ obtained will be isomorphic to $F$. We deduce that \eqref{extension theoreme} is in this case a semi-direct product.

\paragraph*{Step 3: the linear case.}
Let $T$ be a maximal torus of $G$ and let $N=N_G(T)$ be the normalizer of $T$ in $G$. It is a smooth $k$-subgroup by \cite[VI$_\mathrm{B}$, Prop.~6.2.5; XI, Cor.~2.4]{SGA3}. The neutral connected component $N^\circ$ of $N$ is nilpotent (cf.~\cite[XII, Cor.~6.7]{SGA3}), hence it is isomorphic to a direct product $T\times_k U$ with $U$ a unipotent group (cf.~\cite[XVII, Thm.~7.3.1, Thm.~6.1.1]{SGA3}). We will first show that we can pull back \eqref{extension theoreme} to an extension of $F$ by $N$.

Let $(\f,\g)$ be a cocycle representing $\varphi(\xi)$. By Proposition \ref{proposition ext non abelien}, we may assume that $(\f,\g)$ restricted to $\gamma_F(\Gamma)\subset F_\Gamma$ is $(\f_G,1)$. Consider then, for $(f,\sigma)\in F_\Gamma$ the $k_s$-subgroup $\f_{(f,\sigma)}(T)$ of $G$. This is a maximal torus and hence it is conjugate to $T$ over $k_s$ (cf.~for example \cite[XII, Thm.~1.7]{SGA3}), i.e. there exists $c_{(f,\sigma)}\in G(k_s)$ such that $\f_{(f,\sigma)}(T)=c_{(f,\sigma)}^{-1}Tc_{(f,\sigma)}$. Now, from our assumption on $(\f,\g)$ it is easy to see that one may take $c_{(1,\sigma)}=1$ for every $\sigma\in\Gamma$. Since $\gamma_F(\Gamma)$ is clearly open in $F_\Gamma$, this implies that one may choose the $c_{(f,\sigma)}$ in order to get a continuous map $c:F_\Gamma\to G(k_s)$ which is trivial on $\gamma_F(\Gamma)$ and hence, up to changing $(\f,\g)$ by $(c\cdot\f,c\cdot\g)$ (see Definition \ref{defintion F-cocycles}), we may further assume that $\f_{(f,\sigma)}$ fixes $T$ for all $(f,\sigma)\in F_\Gamma$.

It suffices then to look at equation \eqref{equation proprietes 2 cocycles f morphisme a intg pres} to see that $\g$ must take values in $N(k_s)$. And since $\f_{(f,\sigma)}$ stabilizes $T$, it also stabilizes $N$ for all $(f,\sigma)\in F_\Gamma$. Lemma \ref{lemme technique} applies then, giving us the following commutative diagram
\[\xymatrix@R=5mm{
1 \ar[r] & N \ar[d] \ar[r] & H_0 \ar[d] \ar[r] & F \ar@{=}[d] \ar[r] & 1 \\
1 \ar[r] & G \ar[r] & H \ar[r] & F \ar[r] & 1.
}\]
One should note by the way that the $k$-form of $\bar N=N_{k_s}$ given by Lemma \ref{lemme technique} is $N$ itself since the modifications done above on $\f$ do not change its restriction to $\gamma_F(\Gamma)$ and hence do not change the natural action of $\Gamma$. It suffices then to prove the result for this extension $H_0$.\\

Recall now that $N$ sits in an extension
\[1\to T\times_k U\to N\to W\to 1,\]
with $U$ a smooth connected unipotent group. In particular, since $T\times_k U=N^\circ$ is clearly normal in $H_0$, we get an extension
\[1\to N^\circ\to H_0\to F'\to 1,\]
where $F'$ is an extension of $F$ by $W$ and hence a finite group. Consider now this last extension and take quotients
by $U$ and by $T$. We get extensions
\[1\to R\to H_R\to F'\to 1,\]
where $R$ denotes either $T$ or $U$. Steps 1 and 2 tell us then that there exists a finite smooth $k$-subgroup
$S_R\subset R$ and a commutative diagram with exact rows
\[\xymatrix@R=5mm{
1 \ar[r] & S_R \ar[d] \ar[r] & H'_R \ar[d] \ar[r] & F' \ar@{=}[d] \ar[r] & 1 \\
1 \ar[r] & R \ar[r] & H_R \ar[r] & F' \ar[r] & 1.
}\]
Consider then the direct product of $H_T$ and $H_U$. One easily sees that there is a commutative diagram
\[\xymatrix@R=5mm{
1 \ar[r] & S_T\times_k S_U \ar[d] \ar[r] & H'_T\times_k H'_U \ar[d] \ar[r] & F'\times_k F' \ar@{=}[d] \ar[r] & 1 \\
1 \ar[r] & T\times_k U \ar[r] & H_T\times_k H_U \ar[r] & F'\times_k F' \ar[r] & 1.
}\]
Consider now the diagonal subgroup $F'$ of $F'\times_k F'$ and take the preimages of this subgroup in both rows. Since $\aut(N^\circ_{k_s})=\aut(T_{k_s})\times\aut(U_{k_s})$ (both subgroups are characteristic), it is then easy to see that in the lower row we recover our group $H_0$. Setting then $S':=S_T\times_k S_U$ we get the commutative diagram
\[\xymatrix@R=5mm{
1 \ar[r] & S' \ar[d] \ar[r] & H' \ar[d] \ar[r] & F' \ar@{=}[d] \ar[r] & 1 \\
1 \ar[r] & N^\circ \ar[r] & H_0 \ar[r] & F' \ar[r] & 1.
}\]
Define $S$ as the preimage of $W\subset F'$ in $H'$. This group is smooth and finite since $W$ and $S'$ are. We see then that $H'$ is an extension of $F$ by a subgroup $S$ of $N$ pulling back the extension $H_0$, hence pulling back \eqref{extension theoreme}. Moreover, if $G^\circ$ is reductive, then $U=\{1\}$, cf.~\cite[XIX, Lem.~1.6.2]{SGA3}. And if $nw$ is prime to $p$, step 2 tells us that we may assume $S_U$ to be trivial. In both cases, $S$ becomes an extension of $W$ by $S_T$ whose order divides $(ndw)^r$ by step 1. The order of $S$ divides then $(nd)^rw^{r+1}$.

\paragraph*{Step 4: the general case.}
This is simply an application of the previous steps. Indeed, for $G$ an arbitrary smooth $k$-group there is always an
exact sequence
\[1\to G^\circ\to G\to F'\to 1,\]
where $G^\circ$ is connected and characteristic in $G$ and both $F'=G/G^\circ$ and $G^\circ$ are smooth. In particular, extension \eqref{extension theoreme} gives rise to an extension
\[1\to G^\circ\to H\to F''\to 1,\]
where $F''$ is also finite and smooth, since it is an extension of $F$ by $F'$. We easily reduce then to the case where $G$ is connected. Now, for a smooth connected $k$-group $G$, there is a unique exact sequence
\[1\to L\to G\to A\to 1,\]
where $L$ is a smooth connected linear $k$-group and $A$ is an abelian variety. (cf.~\cite{ConradChevalley}). And since $L$ is characteristic in $G$, the same induction process we applied in step 2 works here (take quotients by $L$, use step 1 to deal with $A$, then step 3 to deal with $L$). This concludes the proof.
\end{proof}

\section{Some finiteness results}\label{section finitude}
Notations are as above. Given Proposition \ref{proposition compatibilite ext} and Theorem \ref{theoreme reduction ext}, one can prove the finiteness of $\ext(F,G,\kappa)$ under convenient hypotheses.

Following Serre (cf.~\cite[III.4.2]{SerreCohGal}), we say that a field $k$ is of type (F) if it is perfect and if, for every $n\geq 1$, there exist only a finite number of subextensions of $k_s$ of degree $n$ over $k$. Examples of such fields are $\bb{R}$, $\bb{C}((T))$, finite fields and $p$-adic fields. For such fields, one can state the following result.

\begin{thm}\label{theoreme finitude}
Let $k$ be a field of type {\rm (F)} of characteristic $p\geq 0$, let $F$ be a finite smooth $k$-group of order $n$ and let $G$ be a smooth $k$-group. Assume there is an outer action $\kappa$ of $F$ on $G$. Assume moreover that one of the following holds:
\begin{enumerate}
\item $k$ is finite;
\item $G$ is linear and $n$ is prime to $p$;
\item $G$ has a reductive neutral component.
\end{enumerate}
Then the set $\ext(F,G,\kappa)$ is finite.
\end{thm}

\begin{proof}
By Proposition \ref{proposition compatibilite ext}, we know that $\ext(F,Z,\kappa)$ acts transitively on $\ext(F,G,\kappa)$. It will suffice then to prove the case where $G$ is commutative. Moreover, since the exact sequence
\[1\to G^\circ \to G\to G/G^\circ\to 1,\]
is clearly $F$-equivariant, we can reduce to the case where $G$ is either finite or connected by \cite[XVII, App.~I, Prop.~2.1]{SGA3}.\\

Assume then first that $G$ is finite and commutative. Using Proposition \ref{proposition ext abelien}, we see that we only need to show finiteness of $H^1(k,G)$ and of $\im(\varphi)=\ker(\res)$, where $\res$ is the restriction map $H^2(F,G,\kappa)\to H^2(k,G)$. The group $H^1(k,G)$ is finite by \cite[III.4.1, Prop.~8]{SerreCohGal}. As for $H^2(F,G,\kappa)=H^2(F_\Gamma,G(k_s))$, it can be calculated via the spectral sequence
\[E^{p,q}_2=H^p(\Gamma,H^q(F(k_s),G(k_s)))\Rightarrow H^{p+q}(F_\Gamma,G(k_s))=E^{p+q}.\]
Let us look at the terms with $p+q=2$. The term $E_2^{0,2}$ is finite simply because $F(k_s)$ and $G(k_s)$ are finite groups. This is also the case for $E_2^{1,1}$ by finiteness of $F(k_s)$ and $H^1(k,G)$. Finally, recall that for the term $E_2^{2,0}$ there is a natural morphism
\[E_2^{2,0}=H^2(\Gamma,H^0(F(k_s),G(k_s)))\to H^{2}(F_\Gamma,G(k_s))=E^2,\]
which corresponds to inflation for the quotient map $F_\Gamma\to\Gamma$. Composing this morphism with $\res$, and denoting by $G^F$ the biggest smooth $k$-subgroup of $G$ that is invariant by the action of $F$, this corresponds to the natural morphism $H^2(k,G^F)\to H^2(k,G)$, whose kernel is a quotient of the finite group $H^1(k,G/G^F)$ (this is once again a consequence of \cite[III.4.1, Prop.~8]{SerreCohGal}). One sees then that the part of the image of $E_2^{2,0}$ that actually falls into the kernel of $\res$ is a finite group. This proves that the whole group $\ker(\res)$ is finite.\\

Assume now that $G$ is connected and commutative. In cases 2 and 3, the result is then a corollary of Theorem \ref{theoreme reduction ext} (note that the Weyl group is trivial for $G$ connected and commutative). Indeed, then we know that there exists a positive integer $m$ and a finite smooth $k$-subgroup $S$ of the $m$-torsion of the maximal torus of $G$ such that the map
\[\ext(F,S,\kappa)\to \ext(F,G,\kappa),\]
is surjective. We are then reduced to the finite case which we have already proved.

Finally, in case 1, every extension of $F$ by $G$ induces a natural $G$-torsor over $F$ and hence an element of $H^1_{\text{\'et}}(F,G)$. This element is clearly trivial since $F$ as a variety is a product of separable extensions of $k$ and hence of finite fields, whence the triviality of the whole $H^1_{\text{\'et}}(F,G)$ by Lang's theorem (cf.~\cite{LangCorpsFinis}). We deduce that every extension is a trivial torsor and hence admits a schematic section. Thus $\ext(F,G,\kappa)$ is isomorphic to the Hochschild cohomology group $H^2_0(F,G)$ (cf.~\cite[XVII, App.~I, Prop.~3.1]{SGA3}). Let us show then that this last group is finite in this context.

Recall that Hochschild cohomology groups are obtained as the derived functors of the functor $A\mapsto A^F(k)$ (cf.~\cite[II.3.1.3]{DemazureGabriel}). Let $K/k$ be a finite Galois extension such that $F_K$ is constant. We can see this functor as the composition of three functors
\[A\mapsto A_K\mapsto [A_K^{F_K}(K)=A^F(K)]\mapsto A^F(k),\]
where the last one is simply the functor of fixed-$\Gamma_{K/k}$-points (where $\Gamma_{K/k}$ denotes of course the corresponding Galois group). Since the first of these three functors is clearly exact, Grothendieck's spectral sequence for the other two functors gives then, for our group $G$,
\[H^p(\Gamma_{K/k},H^q_0(F_K,G_K))\Rightarrow H^{p+q}_0(F,G).\]
Now, since $F_K$ is a constant group, we have $H^q_0(F_K,G_K)=H^q(F(K),G(K))$ (cf.~\cite[III.6, Prop.~4.2]{DemazureGabriel}) and hence this group is finite for every $q$ since both $F(K)$ and $G(K)$ are finite for a finite field $K$. The finiteness of $\Gamma_{K/k}$ gives then the finiteness of \emph{every} term of the spectral sequence and hence in particular of $H^2_0(F,G)$.
\end{proof}

\begin{rems}
{\bf 1.} The extra hypothesis for linear $G$ is necessary. Indeed, Proposition \ref{proposition ext abelien} and \cite[III.2.1, Prop.~6]{SerreCohGal} tell us that for $G$ connected abelian and unipotent we have that $\ext(F,G,\kappa)$ is equal to $H^2(F,G,\kappa)$. Take then for example $G=\bb{G}_\mathrm{a}$ and $F=\mathbb{Z}/p\mathbb{Z}$ with trivial action and consider the group $H^0(\Gamma,H^2(F(k_s),G(k_s)))$. Since $F$ is cyclic, we have $H^2(F(k_s),G(k_s))=\hat H^0(F(k_s),G(k_s))=G(k_s)$ and hence $H^0(\Gamma,H^2(F(k_s),G(k_s)))=G(k)=k$, that is, an infinite group if $k$ is not finite. Now, by the spectral sequence used above, this group appears as a quotient of $H^2(F,G,\kappa)$. Indeed, since the $p$-cohomological dimension of $\Gamma$ is $\leq 1$ (cf.~\cite[II.2.2, Prop.~3]{SerreCohGal}), it is easy to see that this infinite group does not disappear in the spectral sequence and hence our $H^2(F,G,\kappa)$ actually surjects onto it.\\
{\bf 2.} This last remark also tells us that the situations in which one can always reduce extensions to the \emph{same} finite $k$-subgroup in Theorem \ref{theoreme reduction ext} cannot be much more general in the case of linear $G$, otherwise we would get a proof of finiteness of $\ext(F,G,\kappa)$ in these cases too by reduction to that particular finite subgroup. One could wonder for example about the case in which $n$ is prime to $p$ but $w$, the order of the Weyl group of $G$, is not. In this case we have proved finiteness of $\ext(F,G,\kappa)$ and hence we should be able to reduce all of these extensions to a certain finite $k$-subgroup, but controlling its order would then be another issue.
\end{rems}

\end{document}